\documentclass[11pt]{amsart}
\usepackage[utf8]{inputenc}

\usepackage{mathrsfs}
\usepackage{parskip}
\usepackage[margin=1in]{geometry}
\usepackage{wrapfig}

\usepackage{enumerate}
\usepackage{amsthm}
\usepackage{amsmath}
\usepackage{amsfonts}
\usepackage{amssymb}
\usepackage{graphicx}
\usepackage{color}
\usepackage{graphics}
\usepackage{eepic}
\usepackage{nicefrac}

\usepackage{tikzsymbols}

\newcommand{\ignore}[1]{}

\usepackage[colorinlistoftodos]{todonotes}

\newcommand{\commE}{\todo[inline, color=green!40]}

\newcommand{\be}{\begin{equation}}
\newcommand{\ee}{\end{equation}}

\renewcommand{\Re}{\operatorname{Re}}
\renewcommand{\Im}{\operatorname{Im}}

\newcommand{\e}{\varepsilon}

\newcommand{\C}{{\mathbb{C}}}
\newcommand{\R}{{\mathbb{R}}}

\newcommand{\p}{\partial}
\newcommand{\ol}{\overline}

\usepackage{mathtools} 
\mathtoolsset{showonlyrefs,showmanualtags}

\newtheorem{thm}{Theorem}[section]

\newtheorem{conj}[thm]{Conjecture}
\newtheorem{prop}[thm]{Proposition}

\newtheorem{lemma}[thm]{Lemma}

\newtheorem{question}[thm]{Question}
\newtheorem{prob}[thm]{Problem}

\theoremstyle{definition}

\theoremstyle{remark}
\newtheorem{remark}[thm]{Remark}

\title[Critical points]{Counterexamples to Magnanini's conjecture concerning critical points of the torsion function}

\author{Parker Edwards}
\address{Department of Mathematics and Statistics, Florida Atlantic University, Boca Raton, USA-33431.}
\email{edwardsp@fau.edu}

\author{Erik Lundberg}
\address{Department of Mathematics and Statistics, Florida Atlantic University, Boca Raton, USA-33431.}
\email{elundber@fau.edu}

\author{Koushik Ramachandran}
\address{Tata Institute of Fundamental Research, Centre for Applicable Mathematics, Bengaluru, India-560065}
\email{koushik@tifrbng.res.in}

\begin{document}

\begin{abstract} 
We construct planar simply-connected domains whose torsion function has an arbitrary prescribed number of local maxima while the distance to the boundary has only one local maximum inside the domain. This disproves a conjecture of Magnanini asserting that the number of local maxima of the torsion function is bounded above by the number of local maxima of the distance function to the boundary. We also answer a question of Steinerberger concerning the eccentricity of level sets near the point where the torsion function attains its maximum in simply-connected domains.
\end{abstract}

\maketitle

\section{Introduction}

Let $\Omega\subset \mathbb{C}$ be a finitely-connected bounded domain,
and consider the solution $v$ to the following boundary value problem
\begin{equation} 
\label{eq:landscape}
\begin{cases}
\Delta v(z) = -2, & z\in\Omega \\
v(z) =0, & z\in\partial\Omega.
\end{cases}
\end{equation}
Since $v$ is superharmonic, it does not have any minima
in $\Omega$ (only maxima and saddles), and due to the vanishing Dirichlet condition it always has at least one maximum in $\Omega$.  
When $\Omega$ is simply-connected, the function $v$ coincides with the \emph{torsion function} of $\Omega$,
used in elasticity theory to model the shear stress
in a cylindrical bar with uniform cross section $\Omega$
undergoing a twisting force \cite{Musk}.
In this context, the critical points of $v$ correspond to the physical locations where no stress is felt.
A natural question is whether the number of critical points of $v$ is controlled by purely geometric features of $\Omega$. In particular, Magnanini made the following conjecture in \cite[Sec. 3.4]{Magnanini}.

\begin{conj}[Magnanini]\label{conj:Magn}
Suppose $\Omega$ is bounded and simply-connected.
The number $N$ of critical points of $v$ in $\Omega$ satisfies
$$ N \leq 2m-1,$$
where $m$ denotes the number of maxima in $\Omega$ of the function $d_{\p \Omega}(z)$ defined as the distance from $z$ to $\p \Omega$.
Equivalently, $M \leq m$ where $M$ denotes the number of maxima of the torsion function.
\end{conj}
The ``equivalently'' part of the conjecture follows from superharmonicity, the Hopf Lemma, and the Poincar\'e-Hopf index theorem which imply that in a simply-connected domain, the number of saddles is one less than the number of maxima \cite[Thm. 3.3]{AlMa}.

We will refer to the maxima (locations where a local maximum occurs) of $v$ as \emph{hot spots}, since the solution to \eqref{eq:landscape} can be interpreted as a steady state temperature distribution under uniform heat source with boundary held at constant temperature, and we refer to the maxima of $d_{\partial \Omega}$ as \emph{far spots}.  Our main result provides counterexamples with arbitrarily many hot spots and only one far spot.

\begin{thm}\label{thm:MagCtr}
Fix an integer $M>1$. There exists a bounded simply-connected planar domain (defined in \eqref{eq:domaindef} below) whose torsion function has at least $M$ local maxima, while the distance function to the boundary has only one local maximum, i.e., it is possible for a domain to have an arbitrary number of hot spots with only one far spot.
\end{thm}

Thus, restricting the number of far spots does not, in general, place any constraint on the number of hot spots. This is reminiscent of the examples of Gladiali and Grossi \cite{GG} of domains having an arbitrary number of hot spots while the curvature of the boundary of the domain changes sign at only two points.  The broader problem of identifying purely geometric features of the domain that govern the number of hot spots remains wide open.
It is well-known that when $\Omega$ is convex 
$v$ has a unique critical point, but little is known in the way of upper bounds on the number of critical points in nonconvex settings other than for special classes of algebraic domains (where the number of hot spots admits an upper bound in terms of the algebraic degree of the defining function of the domain, see \cite{LR} for details).

Let us give a brief intuitive description of the proof of Theorem \ref{thm:MagCtr} which is provided in Section \ref{sec:proofMag} below. The counterexample is constructed as a thin tubular neighborhood of a snake-like curve with $M$ bends, consisting of alternating straight and curved pieces (see Figure \ref{fig:MagCtr5} for a contour plot of an example with $M=5$). The torsion function in such a domain is locally well approximated  at the middle of each straight or curved section by the model case of a strip or annulus, respectively (using Lemmas \ref{lem:straightneck} and \ref{lem:curvedneck} below, which are adaptations of \cite[Lemma 5.2]{LR}).  This is used to obtain thin-domain asymptotics that reveal a curvature-induced increase in the torsion function at each bend, leading to $M$ hot spots. Finally, by slightly tapering the tubular neighborhood, i.e., allowing its radius to shrink (at a sufficiently slow rate to preserve the hot spots), we ensure that the distance function has a unique local maximum, see Lemma \ref{lemma:medialaxis}.  

The critical points of solutions to \eqref{eq:landscape} play an important role in recent studies of localization of eigenfunctions \cite{Mayboroda}, \cite{ArnoldDavidFilJerMay}, 
\cite{ArnoldDavidJerMayFil},
\cite{MayFilclamped}, \cite{Steinerberger},
where the solution $v$ to   \eqref{eq:landscape} plays a prominent role as the so-called ``localization landscape''.  In particular, the hot spots predict locations of high-amplitude activity.  

Our second result addresses a question of Steinerberger about the geometry of the level sets near the maximum of the torsion function.

For planar convex domains, Steinerberger proved that the Hessian matrix of the torsion function evaluated at its point of maximum has maximum eigenvalue that is bounded above by a negative constant depending only on the ratio of the diameter and inradius of the domain.  

\begin{thm}[Steinerberger]\label{thm:Steinerberger}
Let $\Omega$ be a bounded, convex domain, and let $z_0 \in \Omega$ be the point where the torsion function $v$ assumes its maximum. 
There are universal constants $c_1, c_2 > 0$ such that
\[
\lambda_{\max} \left( D^2 v(z_0) \right) \leq -c_1 \exp\left( -c_2 \frac{\operatorname{diam}(\Omega)}{\operatorname{inrad}(\Omega)} \right),
\]
where $\lambda_{\max}$ denotes the largest eigenvalue of the Hessian matrix of $v$ at $z_0$.
\end{thm}

Since both eigenvalues are negative and their sum (which is the trace of the Hessian) is $\Delta v = -2$, this gives a bound on the eccentricity of the (approximately elliptical) level sets in the vicinity of the maximum. Steinerberger \cite[Sec. 3]{Steinerberger} asked whether the convexity assumption is needed.

\begin{question}[Steinerberger]
\label{q:Steinerberger}
Does Theorem \ref{thm:Steinerberger} also hold true on domains that are not convex but merely simply connected or perhaps only bounded?
\end{question}

A partial answer to this question was given in \cite{Chen2021}, namely, providing multiply-connected examples with Hessian matrix arbitrarily close to degenerate at the point of maximum.  This addresses the ``perhaps only bounded'' part of the question, leaving open whether Theorem \ref{thm:Steinerberger} holds true for simply-connected domains.  We give a complete answer to Question \ref{q:Steinerberger} by showing that the classical Neumann oval (also known as Hippopede of Booth) has, for certain parameter values, a degenerate (not just close to degenerate) Hessian at the point where it attains its maximum.

\begin{thm}
[negative answer to Question \ref{q:Steinerberger}]
\label{thm:SteinNeg}
Let $\Omega \subset\mathbb{C}$ be the interior region of the Neumann oval
$$\{(x,y): (x^2 +y^2)^2 \leq a^2(x^2 + y^2) + 4x^2\}.$$
For $a = \sqrt{\frac{2}{1+\sqrt{2}}}$ the torsion function $v$ of $\Omega$ has a unique maximum that occurs at $z=0$, where the Hessian $H$ has rank one and is given by 
$$H = \left( \begin{array}{cc}
   0 & 0 \\
   0 & -2 \\
  \end{array}\right).$$
\end{thm}


The proof of Theorem \ref{thm:SteinNeg} extends the computations from \cite{LR} utilizing an explicit analytic description of the torsion function for the Neumann oval \cite{FlSi}.  We note that these examples are not only simply-connected but also star-shaped (see Remark \ref{rmk:star}), showing further the importance of the convexity assumption in Theorem \ref{thm:Steinerberger}.

\subsection*{Acknowledgements}
The second named author is supported by Simons Foundation grant 712397, which also provided travel support for the third named author. The third named author gratefully acknowledges the hospitality of Florida Atlantic University during his visit when this research was initiated.

\section{Domains with lots of hot spots and only one far spot (counterexamples to Magnanini's conjecture)}\label{sec:proofMag}

Before beginning the proof of Theorem \ref{thm:MagCtr}, we recall the explicit solution for the torsion function of an annulus and take note of an asymptotic for its maximum value in the case of a thin annulus.

Recall that the solution $v(z)$ to \eqref{eq:landscape} in the case $\Omega$ is an annulus $\{ z: r < |z| < R \}$ is given by
$$ v(z) = \frac{R^2 - |z|^2}{2} + A \log\left( |z| / R\right), \quad \text{where } A=\frac{R^2 - r^2}{2 \log(R/r)}.$$

The function $v(z) = v(\rho)$ depends only on the radial coordinate $\rho = |z|$ and has a degenerate circular ridge of maxima.  The radius of the ridge is determined by setting $v'(\rho) = -\rho + A / \rho = 0$ which gives $\rho^2 = A$.  Thus, the height $v^*$ of the ridge (maximum value of $v$) is given by
$$v^* = v\left(\sqrt{A}\right) = \frac{R^2}{2} + \frac{A}{2}\left( \log A - 1 - 2 \log R \right).
$$
Specializing to the case of a thin annulus,  with inner radius $r = r_0 - \e$ and outer radius $R = r_0 + \e$, and expanding in powers of $\e$, we obtain an asymptotic (for small $\e>0$) expansion for the height of the maximal ridge:
\begin{equation}\label{eq:annulusasymp} v^* = \e^2 + \frac{\e^4}{9r_0^2} + O(\e^6), \quad  \text{as } \e \rightarrow 0.
\end{equation}

\begin{figure}[ht]
    \centering
    \includegraphics[scale=0.44]{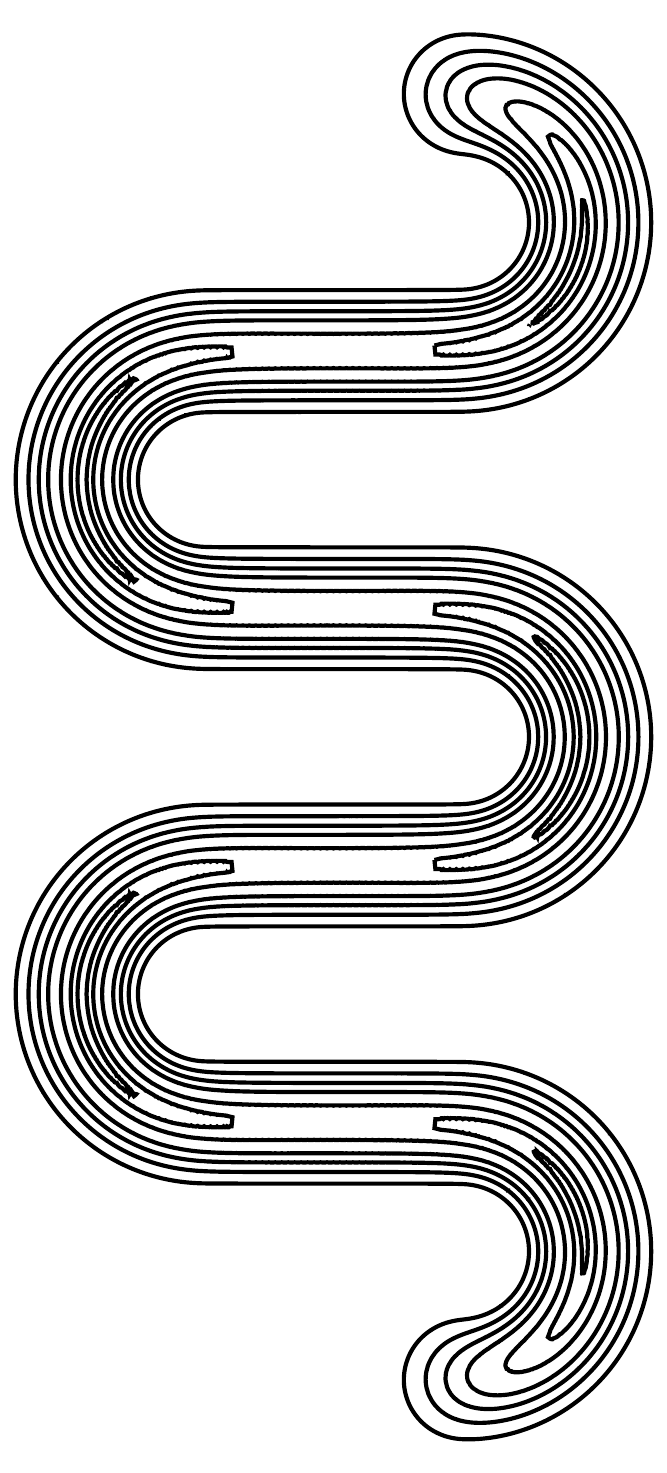}
    \caption{Contour plot of the torsion function of the domain $\Omega_{\e,\delta}$ with $\e = 0.5$ and $\delta=0$ for $M=5$.  The five hot spots are apparent. Introducing a slight taper by taking $\delta>0$ small preserves the hot spots while leaving only one far spot.  
    Although the proof takes $\e>0$ small in order to use thin-domain asymptotics, numerical simulations indicate that the same phenomenon persists for moderately large $\e>0$ as illustrated here.}
    \label{fig:MagCtr5}
\end{figure}

\begin{proof}[Proof of Theorem \ref{thm:MagCtr}]
We construct the domain as a variable-radius tubular neighborhood of a smooth curve $\Gamma$ consisting of the union of the following collection of line segments and semi-circles: For $k=1,\ldots,M-1$ take the horizontal line segments
$[-1 + 2ki,1+2ki]$ in the complex plane. Take the union of these line segments with $M$
semicircles: for
$\ell=0,\ldots,\lfloor (M-1)/2\rfloor$ take the right semicircle of unit
radius centered at $1+(1+4\ell)i$, and for
$\ell=0,\ldots,\lfloor (M-2)/2\rfloor$ take the left semicircle of unit radius
centered at $-1 + (3+4\ell)i$.
By smoothing out a small portion of the circular arcs at each junction where a semi-circle meets a line segment, we may assume $\Gamma$ is $C^3$-smooth, for the purpose of applying Lemma \ref{lemma:medialaxis} below in the last step of the proof.

Let $z(s)$ be an arclength parametrization of $\Gamma$, and define
\begin{equation}\label{eq:domaindef}
\Omega_{\varepsilon,\delta}
=
\bigcup_s B_{r(s)}(z(s)),
\qquad
r(s)=\varepsilon-\delta s.
\end{equation}
Then we will show that, for $\varepsilon\gg\delta>0$ sufficiently small,
$\Omega_{\varepsilon,\delta}$ has the desired properties stated in the theorem; its torsion function has at least $M$ local maxima while the distance to the boundary
$\operatorname{dist}(\cdot,\partial\Omega_{\varepsilon,\delta})$ has only one local maximum.

\subsection*{Step 1 (Abundance of hot spots for $\Omega_{\e,0}$ via thin-tube asymptotics).}

As a first step, we show that the $\delta=0$ case $\Omega_{\e,0}$ satisfies the criterion on the number of hot spots (it has infinitely many far spots, but that will be rectified later).  We will need the following lemmas that will be used to establish contrasting asymptotics in the straight and curved portions of the domain.

\begin{lemma}\label{lem:straightneck}
For \(\e>0\), let
\[
N_\e=(-1,1)\times(-\e,\e),
\]
and suppose \(u_\e:N_\e\to\mathbb{R}\) is harmonic and satisfies
\[
\begin{cases}
u_\e(x,y)=0, & y=\pm\e,\\
|u_\e(x,y)|\le K, & x=\pm1,
\end{cases}
\]
where \(K>0\) is independent of \(\e\). Then
\[
\sup_{y\in[-\e,\e]}|u_\e(0,y)|
\le
\frac{2K}{
\cosh\left(\frac{\pi}{4\e}\right)
}.
\]
In particular, there is a constant $c>0$ such that
\[
\sup_{y\in[-\e,\e]}|u_\e(0,y)|
=
O\!\left(
e^{-c/\e}
\right)
\quad \text{as }\e\to0.
\]
\end{lemma}

\begin{proof}[Proof of Lemma]
Define
\[
h(x,y)
=
C_\e
\cosh\left(\frac{\pi x}{4\e}\right)
\cos\left(\frac{\pi y}{4\e}\right),
\]
where
\[
C_\e
=
\frac{2K}{
\cosh\left(\frac{\pi}{4\e}\right)
}.
\]

Then \(h\) is positive and harmonic in $\R \times(-2\e,2\e).$
Moreover, $
h > 0$ on $y=\pm \e$, and $h > K$ for $x=\pm1$ with $|y| \leq \e$.

Hence, \(|u_\e|\le h\) on \(\partial N_\e\), and by the maximum principle,
\[
|u_\e|\le h
\quad \text{in }N_\e.
\]

In particular,
\[
|u_\e(0,y)|
\le
C_\e
=
\frac{2K}{
\cosh\left(\frac{\pi}{4\e}\right)
}.
\]

Taking the supremum over \(y\in[-\e,\e]\) completes the proof.
\end{proof}

\begin{lemma}\label{lem:curvedneck}
Let
\[
A_\e=\left\{re^{i\theta}: r_0-\e<r<r_0+\e,\ |\theta|<\alpha\right\},
\]
where \(r_0>0\) and \(\alpha>0\) are fixed. Suppose \(u_\e:A_\e\to\mathbb{R}\) is harmonic and satisfies
\[
\begin{cases}
u_\e(z)=0, & r=r_0\pm\e,\\
|u_\e(z)|\le K, & \theta=\pm\alpha,
\end{cases}
\]
where \(K>0\) is independent of \(\e\). Then
\[
\sup_{|r-r_0|<\e}|u_\e(r,0)|
\le
\frac{2K}{
\cosh\left(
\frac{\pi\alpha}{
2\log\left(\frac{r_0+\e}{r_0-\e}\right)
}
\right)
}.
\]
In particular, there is a constant $c>0$ such that
\[
\sup_{|r-r_0|<\e}|u_\e(r,0)|
=
O\!\left(
\exp\left( - c / \e
\right)
\right)
\qquad\mbox{as }\e\to0.
\]
\end{lemma}

\begin{proof}[Proof of Lemma]
Consider the conformal map
\[
\zeta=\log z=x+iy.
\]
Under this map, the annular sector
\[
A_\e=\left\{re^{i\theta}: r_0-\e<r<r_0+\e,\ |\theta|<\alpha\right\}
\]
is sent to the rectangle
\[
R_\e=
\left(
\log(r_0-\e),
\log(r_0+\e)
\right)\times(-\alpha,\alpha).
\]

Define the harmonic (since it is the composition of a harmonic function with a conformal map) function
\[
U(\zeta)=u_\e(e^\zeta),
\]
which satisfies
\[
U=0
\qquad\mbox{on }x=\log(r_0\pm\e),
\]
and
\[
|U|\le K
\qquad\mbox{on }y=\pm\alpha.
\]

Let
\[
\delta_\e
=
\frac12
\log\left(\frac{r_0+\e}{r_0-\e}\right).
\]
Noticing that the rectangle $R_\e$ is transformed to the 
$\delta_\e/\alpha$-neck  
\[
 (-1,1) \times \left(-\frac{\delta_\e}{\alpha},\frac{\delta_\e}{\alpha}\right),
\]
by a linear change of coordinates (a translation in the real direction, a scaling by $1/\alpha$, followed by a rotation by $90^\circ$), we may apply Lemma \ref{lem:straightneck} which gives
\[
\sup_{x\in[-\delta_\e,\delta_\e]}|U(x,0)|
\le
\frac{2K}{
\cosh\left(
\frac{\pi\alpha}{4\delta_\e}
\right)
}.
\]

Since
\[
4\delta_\e
=
2\log\left(\frac{r_0+\e}{r_0-\e}\right),
\]
we obtain
\[
\sup_{x\in[-\delta_\e,\delta_\e]}|U(x,0)|
\le
\frac{2K}{
\cosh\left(
\frac{\pi\alpha}{
2\log\left(\frac{r_0+\e}{r_0-\e}\right)
}
\right)
},
\]
which gives the desired estimate, since $U(x,0) = u_\e(e^x,0)$.
\end{proof}

In order to state the next proposition, let 
\be\label{eq:defrho}
\rho_\e := \sqrt{2\e/\log\left( \frac{1+\e}{1-\e} \right)}
\ee 
which is the radius of the maximal ridge of the torsion function (see the preliminary discussion at the beginning of this section) for the annulus having inner radius $r = 1-\e$ and outer radius $R = 1+\e$.

\begin{prop}\label{prop:Omega}
Let $v$ be the torsion function of $\Omega = \Omega_{\e,0}$.
For all $\e>0$ sufficiently small, we have
$$v(z^*_k) > \sup_{\{ \Re (z) = 0 \}} v(z), \quad k = 0,1,2,...,M-1,$$ 
where $z^*_k = (-1)^{k}(1+\rho_\e) + (1 + 2k)i$ with $\rho_\e$ as defined in \eqref{eq:defrho}. (The points $z^*_k$ converge to midpoints of the semi-circular arcs of $\Gamma$ as $\e \rightarrow 0$.)
In particular, $v$ has at least $M$ local maxima.
\end{prop}

\begin{proof}[Proof of Proposition]
The proof uses the torsion function 
$$v_s(z) = \e^2 - y^2$$ 
of an infinite strip $\{ |\Im (z)| < \e \}$ to approximate $v$ in the straight portions of the domain (near the imaginary axis, to be specific) as well as the torsion function
$$
v_a(z)
=
\frac{(1+\e)^2-|z|^2}{2}
+
\frac{2\e}{\log\!\left(\frac{1+\e}{1-\e}\right)}
\log\!\left(\frac{|z|}{1+\e}\right)
$$
of an annulus with inner radius $1-\e$ and outer radius $1+\e$ (see the preliminary discussion at the beginning of this section) as a model case to approximate $v$ in the curved portions of the domain (at the points $z^*_k$, to be specific).

Fix an arbitrary line segment in $\Omega \cap \{ \Re (z) = 0 \}$ which has midpoint $(2k)i$ for some $k=1,...,M-1$.
Applying Lemma \ref{lem:straightneck} to the harmonic function
$v(z+(2k)i) - v_s(z)$ which is defined in the rectangle $(-1,1) \times (-\e,\e)$ gives the uniform bound $v(z+(2k)i) - v_s(z) = O \left( e^{-c/\e} \right)$ for $z$ on the imaginary axis.  Then using the bound $v_s \leq \e^2$, we have
\begin{equation}
\label{eq:unifstraight}
\sup_{\{ \Re (z) = 0 \}} v(z) \leq \e^2 + O \left( e^{-c/\e} \right),
\end{equation}


Similarly, applying Lemma \ref{lem:curvedneck} to $v(z+(-1)^k+(2k+1)i)-v_a(z)$ which is defined in \[
A_\e=\left\{re^{i\theta}: 1-\e<r<1+\e,\ |\theta|< \pi/2 \right\},
\] we get the estimate $v(z+(-1)^k+(2k+1)i)-v_a(z) = O(e^{-c/\e})$ for $z$ in $A_\e$ with argument $\theta = 0$. Evaluating this at $(-1)^k\rho_\e$ and using the asymptotic \eqref{eq:annulusasymp} (with $r_0 = 1$), we have
\begin{equation}
\label{eq:curvedasymp}
v(z^*_k)= v_a(\rho_\e) + O(e^{-c/\e}) = \e^2 + \frac{\e^4}{9} + O(\e^6),
\end{equation}
where the exponentially small error term provided by the lemma has been absorbed into the $\e^6$-order error term.

Since the asymptotic \eqref{eq:curvedasymp} exceeds the estimate \eqref{eq:unifstraight} for all $\e>0$ sufficiently small, we conclude that $v$ has at least $M$ maxima, namely, at least one occurring in each of the $M$ component regions of $\Omega \setminus \{ \Re (z) = 0 \}$.  Indeed, fixing one such region $U$, we have by compactness that $v$ attains a maximum on the closure $\ol{U}$. The maximum must occur in the interior of $U$, since the value $v(z^*_k)$ is uniformly larger than the boundary values on $\partial U$, which follows from the estimate \eqref{eq:unifstraight} and the vanishing Dirichlet condition on $\partial \Omega$ in the original boundary value problem \eqref{eq:landscape} defining the torsion function $v$.
\end{proof}

\subsection*{Step 2 (preservation of hot spots in the perturbation $\Omega_{\e,\delta}$ of $\Omega_{\e,0}$)}

The domain $\Omega_{\e,0}$ has the desired number of hot spots, but it has an infinite set (namely, the curve $\Gamma$) of far spots.  This will be rectified by passing to the perturbed domain $\Omega_{\e,\delta}$ that has only one far spot, as we will show in Step 3.  First, we check that the hot spots survive under this perturbation when $\delta >0$ is sufficiently small.
For this, we will use the following continuity lemma.  We note in passing that related results can be shown for more general elliptic PDE using Sobolev space methods, but here we opt for a classical approach using the Harnack convergence theorem.

\begin{lemma}
\label{lemma:cvgc}
Let
$
\Omega_1\subset \Omega_2\subset \cdots
$
be an increasing sequence of bounded domains in $\mathbb{R}^2$ whose union
$
\Omega := \bigcup_{j=1}^{\infty}\Omega_j
$
is bounded. For each $j$, let $v_j$ denote the torsion function of $\Omega_j$, that is,
$$
\begin{cases}
\Delta v_j = -2  \text{ in }\Omega_j, \\
\quad v_j =0  \text{ on }\partial\Omega_j,
\end{cases}
$$
and let $v$ denote the torsion function of $\Omega$. Then
$
v_j\longrightarrow v
$
uniformly on compact subsets of $\Omega$.
\end{lemma}

\begin{proof}[Proof of Lemma]
By domain monotonicity of the torsion function,
$$
v_j\leq v_{j+1} \leq v
$$
pointwise in $\Omega_j$.
 Thus, for every $z\in\Omega$, the (eventually defined) sequence $v_j(z)$ is  increasing, and bounded above.
Consequently, the sequence of harmonic functions defined by
$$
h_j(z):=v_j(z)+\frac{|z|^2}{2}.
$$ is also pointwise increasing and bounded by $v(z)+\frac{|z|^2}{2}$.  By the Harnack convergence theorem, $h_j$ converges uniformly on compact subsets of $\Omega$ to a harmonic function $h$. Hence
$$
v_j\longrightarrow w:=h-\frac{|z|^2}{2}
$$
uniformly on compact subsets of $\Omega$, and
$ \Delta w = -2 $.

Since $v_j\leq v$, we have $w\leq v$, and hence $v-w$ is a harmonic function in $\Omega$ satisfying
$
0\leq v-w \leq v.
$
The maximum principle therefore gives $ v-w \equiv0$. 
We conclude that
$$
v_j\longrightarrow v
$$
uniformly on compact subsets of $\Omega$, as desired.
\end{proof}

Fix $\e>0$ sufficiently small that the conclusion of Proposition \ref{prop:Omega} holds, and let $v^*_k$ (for $k=0,1,2,...,M-1$) be as in the proposition statement.

Consider the monotone sequence of domains $\Omega_j := \Omega_{\e,\delta_j}$ with $\delta_j$ decreasing toward zero, and let $v_j$ denote the torsion function of $\Omega_j$.
By domain monotonicity, we have $v_j \leq v$.  In particular this holds along the imaginary axis, where we have the uniform bound $v_j \leq L_0:=\sup_{\{ \Re (z) = 0 \}} v(z)$.
On the other hand, applying Lemma \ref{lemma:cvgc} at each of the $M$ points $z^*_k$, $k = 0,1,2,...,M-1$, we have $v_j(z^*_k) \rightarrow v(z^*_k)$.  Hence, for all $j$ sufficiently large, we have $v_j(z^*_k) > L_0$ for all $k = 0,1,2,...,M-1$, and we conclude as in the proof of Proposition \ref{prop:Omega} that the torsion function $v_j$ of $\Omega_j := \Omega_{\e,\delta_j}$ has at least $M$ maxima for all $j$ sufficiently large.

\subsection*{Step 3 ($\Omega_{\e,\delta}$ has a unique far spot)}

Recall that the \emph{medial axis} $\text{Med}(S)$ of a closed subset $S\subseteq\mathbb{R}^n$ is the set
\[ 
\text{Med}(S) := {\{z \in\mathbb{R}^n \setminus S \mid \exists p_1\not= p_2\in S, \text{dist}(z,S) = \vert z - p_1\vert = \vert z-p_2\vert \}}\text{.}
\]
One has immediately that if $z\not\in\text{Med}(S)$ then $z$ is not a local maximum of the distance function $\text{dist}(\cdot,S)$, as one can increase distance in the opposite of the direction to the unique distance-minimizer to $z$ on $S$. For this step, it is therefore sufficient to characterize the relevant part of the medial axis for the boundary, $\partial\Omega_{\varepsilon,\delta}$, of our proposed counterexample(s) and the relevant distance function when restricted to that medial axis. 

\begin{lemma}\label{lemma:medialaxis}
Let $\Gamma$ be a $C^3$ curve with arclength parametrization
$z(s)$, $s\in[0,L]$.
Define
\[
A:=\max_{s\in[0,L]} |z'''(s)|,
\]
and
\begin{equation}
d_A
:=
\inf_{\substack{s,t\in[0,L]\\ |s-t|\ge 1/\sqrt A}}
|z(s)-z(t)|.
\end{equation}
Assume $d_A > 0$.  Fix $\e>0$ satisfying 
\(
\varepsilon<\min\left\{\frac{1}{16\sqrt A},\frac{d_A}{4}\right\}.
\)
For all sufficiently small $\delta>0$, the curve $\Gamma$ is part of the medial
axis of the boundary of the variable-radius tube
\[
\Omega_{\varepsilon,\delta}
=
\bigcup_{0\le s\le L}
B\bigl(z(s),r(s)\bigr),
\qquad
r(s)=\varepsilon - \delta s.
\]
which intersects $\Omega_{\e,\delta}$.
Moreover,
\[
\operatorname{dist}(z(s),\partial\Omega_{\varepsilon,\delta})=r(s),
\qquad 0\le s\le L.
\]
In particular, the distance to the boundary $\operatorname{dist}(\cdot,\partial\Omega_{\varepsilon,\delta})$ restricted to $\Omega_{\e,\delta}$ has a unique local maximum, attained at $z(0)$.
\end{lemma}

It is possible that this result is known (perhaps even with a relaxed smoothness assumption).  However, as we were unable to find a reference, we include a proof.  The proof is based on elementary vector calculus and point-set topology.

\begin{proof}[Proof of Lemma]
Since $z$ is an arclength parametrization, we have, $z' = T$, and $z'' = \kappa N$, where $T, N$ are unit tangent and unit normal vectors, respectively, and $\kappa$ is the curvature.
We then have
\[
z'''=\kappa' N-\kappa^2T.
\]
Hence, expanding $|z'''|^2$ as a dot product using the above equation, the assumed bound on $z'''$ implies
\[
|\kappa|\le \sqrt A.
\]
Assume $\delta^2 < 1$ and $\delta L < \epsilon$, then define
\begin{equation}\label{eq:wparam}
w_\pm(s)
=
z(s)+r(s)\delta T(s)
\pm r(s)\sqrt{1-\delta^2}\,N(s),
\end{equation}
where $
r(s)=\varepsilon-\delta s,
$ as stated in the hypothesis of the lemma.

\medskip

\noindent\textbf{Claim.}
The boundary of $\Omega_{\varepsilon,\delta}$ consists of the pair of curves parameterized by
\(
w_\pm(s)
\)
together with the circular arc of $\partial B(z(0),r(0))$ from $w_-(0)$ to $w_+(0)$ along with the circular arc of $
\partial B(z(L),r(L))$
from $w_+(L)$ to $w_-(L)$.

\smallskip

We first note that $\partial \Omega_{\varepsilon,\delta} \subset 
\bigcup_{0\le s\le L}
\partial B\bigl(z(s),r(s)\bigr)$ follows from continuity of $z(s)$ and $r(s)$ along with compactness of $[0,L]$.  This implies 
\begin{equation}\label{eq:bdryOmega}
\partial \Omega_{\varepsilon,\delta} = 
\bigcup_{0\le s\le L}
\partial B\bigl(z(s),r(s)\bigr) \setminus \bigcup_{0\le s\le L}
B\bigl(z(s),r(s)\bigr)
\end{equation}
as the reverse inclusion follows trivially from $\mathring{\Omega}_{\e,\delta} = \Omega_{\varepsilon,\delta}$.

Now we will show that $w_\pm(s) \in \partial \Omega_{\varepsilon,\delta}$.
Since $|w_\pm(s)-z(s)|=r(s)$, we have $w_\pm(s) \in \partial B_{r(s)}(z(s))$.  Hence, in light of  \eqref{eq:bdryOmega}, it suffices to show 
\begin{equation}\label{eq:goalgtr}
|w_\pm(s)-z(t)|>r(t),
\qquad t\neq s.
\end{equation}

Local Case: Let \(h=t-s\), and assume $0< |h| < 1 / \sqrt{A}$.  Taylor's theorem gives
\[
z(t)
=
z(s)
+hT(s)
+\frac{h^2}{2}z''(s)
+R_3,
\qquad
|R_3|
\le
\frac{A}{6}|h|^3.
\]
Substituting this expansion into
\(
|w_\pm(s)-z(t)|^2-r(t)^2
\),
and recalling that $r(t) = r(s) - h\delta$, the constant and linear terms cancel, and the quadratic term is
\[
\left(
1-\delta^2
\mp r(s)\sqrt{1-\delta^2}\,\kappa(s)
\right)h^2.
\]

For \(\delta>0\) sufficiently small (so that $\delta L < \e$ and $\delta^2 < 1/8$), we have 
\begin{equation}
\label{eq:3over4}
(1-\delta^2
\mp r(s)\sqrt{1-\delta^2}\,\kappa(s))h^2
 > \frac34h^2,
 \end{equation}
since
$|r(s)| < \e + \delta L$, $\varepsilon<\frac{1}{16\sqrt A}$,
and $|\kappa(s)|\le\sqrt A$.

Hence
\begin{equation}
|w_\pm(s)-z(t)|^2-r(t)^2 > \frac{3}{4} h^2 + E,
\end{equation}
where $E$ consists of the higher order terms involving \(R_3\) 
\[
E = -2\left(w_\pm(s)-z(s)-hT(s)-\frac{h^2}{2}z''(s)\right)\cdot R_3
+|R_3|^2.
\]
By the Cauchy--Schwarz inequality and triangle inequality, 
\[ |E| \leq 
2\left(r(s)+|h|+\frac{\sqrt A}{2}h^2\right)|R_3|
+|R_3|^2.
\]

Using
\[
|R_3|\le \frac{A}{6}|h|^3,
\qquad
|h|<\frac1{\sqrt A},
\]
we estimate
\[
\begin{aligned}
2\left(r(s)+|h|+\frac{\sqrt A}{2}h^2\right)|R_3|
+|R_3|^2
&\le
2\left(\frac{1}{16 \sqrt{A}}+\frac{1}{\sqrt{A}}+\frac{1}{2\sqrt{A}}\right)
\frac{A}{6}|h|^3
+\frac{A^2}{36}|h|^6  \\
&\le
\left(\frac{1}{8}+1+\frac{1}{2}\right)\frac{1}{3}|h|^2
+\frac{1}{36}|h|^2  \\
&=
\frac{41}{72}|h|^2 < \frac{3}{4}|h|^2.
\end{aligned}
\]
Therefore, for $0< |t-s|< 1 / \sqrt{A}$, we have
\[
|w_\pm(s)-z(t)|^2-r(t)^2>0,
\]
which gives the desired estimate \eqref{eq:goalgtr}.

Nonlocal Case: If instead $|t-s|\ge1/\sqrt A$, then by definition of $d_A$,
\[
|z(s)-z(t)|\ge d_A.
\]
Hence (still assuming $\delta$ sufficiently small so that $\delta L < \e$)
\[
\begin{aligned}
|w_\pm(s)-z(t)|
&\ge |z(s)-z(t)|-|w_\pm(s)-z(s)|\\
&\ge d_A-r(s)\\
&> d_A-2\varepsilon\\
&> 2\varepsilon\\
&> r(t),
\end{aligned}
\]
where we have also used \(
\varepsilon<\frac{d_A}{4}.
\)
This proves the desired estimate \eqref{eq:goalgtr} and shows that $w_\pm(s) \in \partial \Omega_{\varepsilon,\delta}$ for each $s$.

Moreover, for each $s \in (0,L)$, the two points $w_\pm(s)$ are the only points on $\partial B_{r(s)}(z(s))$ that contribute to the boundary of $\Omega_{\varepsilon,\delta}$, since other points lead to a nonzero first-order coefficient in the corresponding Taylor expansion.
Indeed, let $z^*(s) \in \partial B_{r(s)}(z(s))$ and write $z^*(s) = z(s) + r(s) \left( a T(s) + b N(s) \right) $, $a^2 + b^2 = 1$. The first-order coefficient of $
|z^*(s)-z(t)|^2-r(t)^2
$ is $-2r(s)(a+\delta)$
which is nonzero for 
$z^*(s) \neq w_\pm(s)$.
Therefore $z^*(s)$ lies in some nearby disk $B(z(t),r(t))$ and hence does not belong to $\partial\Omega_{\varepsilon,\delta}$.

On the other hand, the disks centered at the endpoints do contribute additional boundary points since there are neighboring disks on only one side.

At $s=0$, let
$$
z^*
=
z(0)+r(0)\left(aT(0)+bN(0)\right),
\qquad
a^2+b^2=1.
$$

For $h = t - s = t > 0$, the linear term in the same expansion of $|z^*-z(h)|^2-r(h)^2$ as above is $
2r(0)(\delta-a)h $ 
which changes sign at $a=\delta$. Points
with $a>\delta$ lie inside nearby disks.  For $a\le\delta$,
the linear term is nonnegative.  The quadratic term is now
$$
\left( 1-\delta^2-r(0)\kappa(0)b \right) h^2,
$$
which admits the same lower bound as in \eqref{eq:3over4}, namely, we have
$$
\left( 1-\delta^2-r(0)\kappa(0)b \right) h^2 > \frac{3}{4} h^2,
$$
since $|b|\le1$, $|\kappa(0)|\le\sqrt A$, and
$r(0)<2\varepsilon<1/(8\sqrt A)$.  Hence, the same local and nonlocal
estimates as in the above interior point case show that
\[
|z^*-z(t)|>r(t)
\qquad\text{for all }t\in(0,L].
\]
Consequently, the points in $\partial B(z(0),r(0))$ that belong to
$\partial\Omega_{\varepsilon,\delta}$ are precisely those with $a\le\delta$, i.e., the circular arc from $w_-(0)$ to $w_+(0)$.

Similarly, at $s=L$
we find that the circular arc of $\partial B(z(L),r(L))$ from $w_-(L)$ to $w_+(L)$ contributes to the boundary of $\Omega_{\varepsilon,\delta}$.

We conclude that $\Gamma$ is contained in the medial axis of \(\partial\Omega_{\varepsilon,\delta}\) since each point $z(s)$ has at least two closest points $w_\pm(s)$ on $\partial \Omega_{\varepsilon,\delta}$. 
It remains to show that there are no other medial-axis points.  Let
$z^*\in\Omega_{\varepsilon,\delta}\setminus\Gamma$, and let $q\in\partial\Omega_{\varepsilon,\delta}$ be a nearest boundary point, i.e., a point realizing the distance
\[
d=\operatorname{dist}(z^*,\partial\Omega_{\varepsilon,\delta})=|z^*-q|.
\]
Then $\overline{B(z^*,d)}\subset\overline{\Omega_{\varepsilon,\delta}}$ and the two circles
$\partial B(z^*,d)$ and $\partial B(z(s),r(s))$, where
$q\in\partial B(z(s),r(s))$, are tangent at $q$ and lie on the same
side of their common tangent.  Hence
\[
\overline{B(z^*,d)}\subset\overline{B(z(s),r(s))}.
\]
Thus every point of $\partial B(z^*,d)$ other than $q$ lies in
$B(z(s),r(s))\subset\Omega_{\varepsilon,\delta}$, so $q$ is the unique closest boundary
point to $z^*$.  Hence $\Omega_{\varepsilon,\delta}\setminus\Gamma$ contains no medial-axis
points, and therefore the medial axis of $\Omega_{\varepsilon,\delta}$ is $\Gamma$.

Moreover, since we have established that $w_\pm(s)$ are nearest points to $z(s)$ for each $s \in [0,L]$, it follows that $$
\operatorname{dist}(z(s),\partial\Omega_{\varepsilon,\delta})
=
r(s),
$$
which has a unique local maximum at $s=0$ (recall $r(s)=\varepsilon - \delta s$).
\end{proof}

Finally, to complete the proof of the theorem, having chosen $\e > \delta >0$ sufficiently small as in Step 2, we take $\e>0$ (and subsequently $\delta>0$) smaller if necessary in order to apply Lemma \ref{lemma:medialaxis}. We conclude that $\Omega_{\e,\delta}$ has at least $M$ hot spots and only one far spot, as desired.
\end{proof}

\section{A simply-connected domain with a degenerate maximum}

\begin{proof}[Proof of Theorem \ref{thm:SteinNeg}]
As stated in the theorem, we consider the one-parameter family of domains $\Omega\subset\mathbb{C}$ with boundary given by the following algebraic curve (excluding the origin, which is an isolated point in the solution set not included in the domain boundary)
$$\{x+iy \in \C: (x^2 +y^2)^2 = a^2(x^2 + y^2) + 4x^2\},$$
where $a>0$ is a parameter.
Our goal is to show that for $a = \sqrt{\frac{2}{1+\sqrt{2}}}$ (see Figure \ref{fig:Neumann}), the torsion function $v$ has a unique maximum that occurs at $z=0$, and the Hessian matrix $H_v$ of the torsion function $v$ is degenerate at this point.  Namely, we want to show
$$H_v(0) = \left( \begin{array}{cc}
   0 & 0 \\
   0 & -2 \\
  \end{array}\right).$$
Let $R = \frac{a + \sqrt{a^2 + 4}}{2}.$ Note that $R>1.$ 

We recall the following facts \cite{FlSi}, \cite{LR}:

\begin{itemize}
\item The complex algebraic function $\phi(z) = \frac{1-R^4 + \sqrt{(R^4-1)^2 + 4R^4z^2}}{2Rz}$, for an appropriate branch of the square root function, maps $\Omega$ conformally onto the unit disk $\mathbb{D}$.
\item The torsion function of $\Omega$ can be expressed explicitly in terms of the conformal mapping $\phi$ as
\begin{equation}
v(z) = \dfrac{R^4 - 1}{2R^2} + \dfrac{\Re(z\phi(z))}{R} - \dfrac{|z|^2}{2}.
\end{equation}
\item The gradient $\nabla v$ (viewed as a complex number) is given by
\begin{equation}\label{grad}
\nabla v(z) = \overline{\dfrac{2R^2z}{\sqrt{(R^4-1)^2 + 4R^4z^2}}} - z.
\end{equation}
\item The number of critical points of $v$ bifurcates at the critical parameter $R_0 = \sqrt{1+\sqrt{2}}$.  Namely, if $1< R < R_0,$   $\nabla v$ has exactly 3 zeros (counting multiplicity), and for $R \geq R_0$ the gradient $\nabla v$ has a single zero, namely, the origin.
\end{itemize}

With all these facts in hand, we proceed to compute the Hessian matrix of $v$ at $z=0$ while taking the critical parameter value $R=R_0 := \sqrt{1+\sqrt{2}}$.
The first row of the Hessian matrix can be expressed as $\partial_x \nabla v = \left( \partial_z  + \partial_{\bar{z}}  \right) \nabla v$, and the second row of the Hessian matrix can be expressed as $\partial_y \nabla v = i\left( \partial_z v - \partial_{\bar{z}}\right) \nabla v$.

We have
\begin{align}
\left( \partial_z  + \partial_{\bar{z}}  \right) \nabla v &= -1 + \frac{2R_0^2- 2R_0^2z((R_0^4-1)^2+4R_0^4z^2)^{-1/2}8R_0^4z}{(R_0^4-1)^2+4R_0^4z^2} \big\rvert_{z=0} \\
&= -1 + \frac{2R_0^2}{(R_0^4-1)} = 0.
\end{align}
This verifies that both entries in first row of the Hessian matrix are equal to zero.
Turning to the second row, we have
\begin{align}
i\left( \partial_z  - \partial_{\bar{z}}  \right) \nabla v &= -i - i\frac{2R_0^2- 2R_0^2z((R_0^4-1)^2+4R_0^4z^2)^{-1/2}8R_0^4z}{(R_0^4-1)^2+4R_0^4z^2} \big\rvert_{z=0} \\
&= -i - i\frac{2R_0^2}{(R_0^4-1)} = -2i.
\end{align}

Hence, for the critical parameter value $R=R_0$, the Hessian matrix of $v$ at $z=0$ is given by
$$H_v(0) = \left( \begin{array}{cc}
   0 & 0 \\
   0 & -2 \\
  \end{array}\right),$$
  as desired.
\end{proof}

\begin{figure}[h]
    \centering
    \includegraphics[scale=0.4]{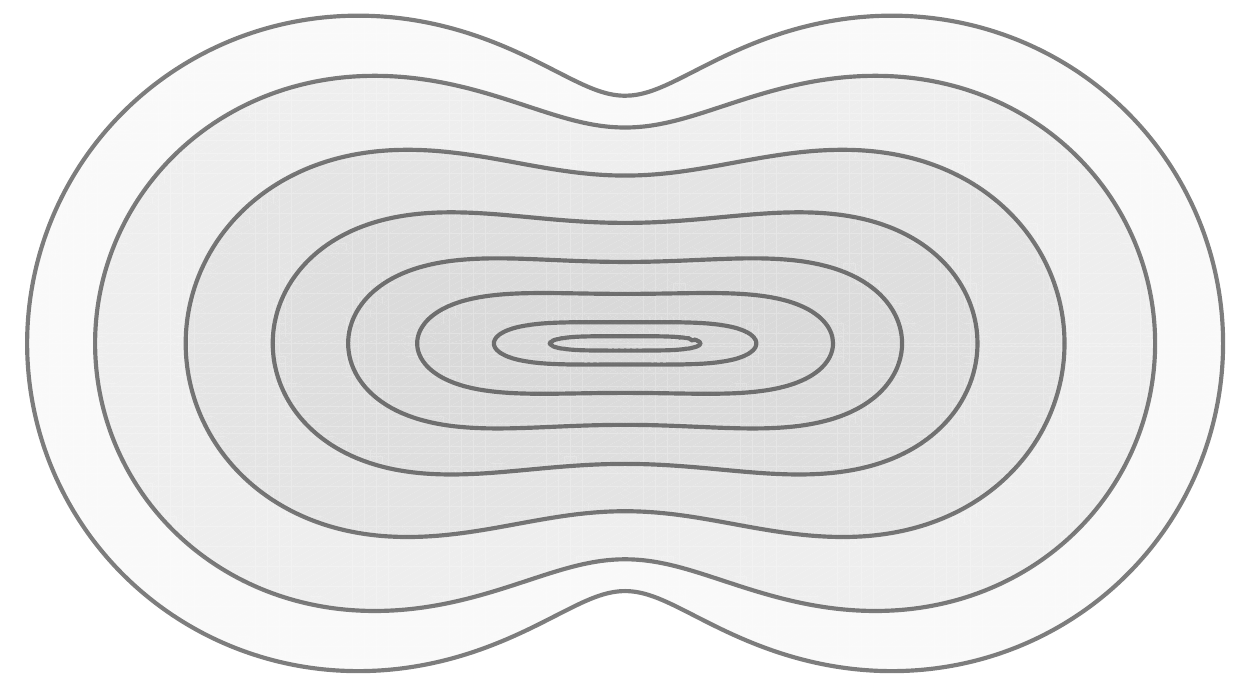}
    \caption{The Neumann oval at its critical parameter along with the level curves of its torsion function which has a degenerate maximum at the center.}
    \label{fig:Neumann}
\end{figure}

\begin{remark}\label{rmk:star}
The Neumann's oval is an example from the family of algebraic curves referred to as Hippopedes. From the classical theory of plane curves it is known that Hippopedes can be parametrized in polar coordinates \cite[p. 145]{La}, showing that the Neumann's oval is star-shaped with respect to the origin. Hence, beyond showing that ``simply-connected'' is insufficient for an eccentricity bound on level curves near the maximum, this example shows that even ``star-shaped'' is insufficient.
\end{remark}

\bibliographystyle{abbrv}
\bibliography{crit}

\end{document}